\documentclass[12pt,twoside,reqno]{amsart}
\usepackage{verbatim,amscd,epsfig,amsmath,amsthm,amstext,amssymb,hyperref,extpfeil,todonotes,enumerate,  
}
\usepackage[all]{xy}
\theoremstyle{plain}
\newtheorem{Theorem}{Theorem}[section]
\newtheorem{Lemma}[Theorem]{Lemma}

\newtheorem{Proposition}[Theorem]{Proposition}

\newtheorem{Definition}[Theorem]{Definition}

\newcommand{\unity}{{1\!\!\!\:\mathrm{l}}}

\newcommand{\Ass}{\mathcal A}

\newcommand{\Mss}{\mathcal M}

\newcommand{\Oss}{\mathcal O}

\newcommand{\Rss}{\mathcal R}
\newcommand{\Sss}{\mathcal S}

\newcommand{\ci}{i}
\newcommand{\C}{\mathbb{C}}
\newcommand{\R}{\mathbb{R}}

\newcommand{\N}{\mathbb{N}}

\DeclareMathOperator{\ad}{ad}

\newcommand{\ind}[1]{_{\mbox{\scriptsize{#1}}}}
\DeclareMathAlphabet{\Mat}{U}{eur}{m}{n} 
\DeclareMathAlphabet{\Set}{U}{eur}{m}{n} 
\newcommand{\Sh}[1]{\mathcal{#1}} 

\begin{document}
\author[S.~Klein]{Sebastian Klein}
\email{s.klein@math.uni-mannheim.de}
\address{Mathematics Chair III\\
Universit\"at Mannheim\\
D-68131 Mannheim, Germany}
\author[E.~L\"ubcke]{Eva L\"ubcke}
\email{eva.luebcke@gmail.com}
\address{Mathematics Chair III\\
Universit\"at Mannheim\\
D-68131 Mannheim, Germany}
\author[M.~Schmidt]{Martin Ulrich Schmidt}
\email{schmidt@math.uni-mannheim.de}
\address{Mathematics Chair III\\
Universit\"at Mannheim\\
D-68131 Mannheim, Germany}
\author[T.~Simon]{Tobias Simon}
\email{tsimon@mail.uni-mannheim.de}
\address{Mathematics Chair III\\
Universit\"at Mannheim\\
D-68131 Mannheim, Germany}
\title{Burchnall-Chaundy Theory}
\begin{abstract}
The Burchnall-Chaundy theory concerns the classification of all pairs of commuting ordinary differential operators. We phrase this theory in the language of spectral data for integrable systems.

In particular, we define spectral data for rank 1 commutative algebras $A$ of ordinary differential operators. We solve the inverse problem for such data, i.e.~we prove that the algebra $A$ is (essentially) uniquely determined by its spectral data. The isomorphy type of $A$ is uniquely determined by the underlying spectral curve. 
\end{abstract}
\maketitle

\section{Introduction}

The Burchnall-Chaundy theory (\cite{BC1,BC2,BC3}) concerns the classification of all pairs $(P,Q)$ of commuting ordinary differential operators $P$ and $Q$ of order $m$ and $n$, respectively. Burchnall and Chaundy carried out their work before the relationship between integrable systems and Riemann surfaces or complex curves, i.e.~the spectral theory for integrable systems, was discovered in the course of the investigation of the integrable system defined by the Korteweg-de~Vries equation. 

The main purpose of the present article is to rephrase the Burchnall-Chaundy theory in terms of the theory of spectral data for integrable systems. For this the Krichever construction and the theory of Baker-Akhiezer functions will play important roles. Because the spectral curve can have singularities, we will use the description of these concepts for analytic singular curves in \cite{KLSS}. Our construction can also serve as an explanation of the relation of these two concepts to differential operators.

We will consider algebras $A$ that are generated by pairs $(P,Q)$ of commuting differential operators of order $m$, $n$. In Section~\ref{Se:direct} we will associate to any such algebra a holomorphic matrix-valued function $M: \mathbb{C} \to \mathbb{C}^{m\times m},\,\lambda\mapsto M(\lambda)$, and thereby spectral data which are composed of a generally singular complex curve $X'$ describing the eigenvalues of $(M(\lambda))_{\lambda\in\mathbb{C}}$ and a second datum describing the corresponding eigenvector bundle. In contrast to parts of Burchnall's and Chaundy's original theory, we will here restrict ourselves to the case where the orders $m$ and $n$ of the generating differential operators are relatively prime. This restriction corresponds to the algebra $A$ being of rank $1$, which means by definition that the eigenspaces of $M(\lambda)$ are generically $1$-dimensional. In this situation the eigenvectors of $M(\lambda)$ comprise a holomorphic line bundle $\Lambda'$ on $X'$. 

Whenever $X'$ has no singularities and is therefore a Riemann surface, the well-known 1--1 correspondence between line bundles and divisors on Riemann surfaces is often used to define spectral data as the pair $(X',D)$ of the eigenvalue curve $X'$ and the divisor $D$ corresponding to the eigenline bundle \,$\Lambda'$\,. In order to define spectral data in a similar manner also for complex curves $X'$ with singularities, the concept of a divisor needs to be generalised. The proper generalisation to use in this context is the concept of \emph{generalised divisors} introduced by Hartshorne \cite{Ha86}, at first for Gorenstein curves. In this sense, a generalised divisor on $X'$ is a subsheaf of the sheaf of meromorphic functions on $X'$ which is locally finitely generated over the sheaf of holomorphic functions on $X'$. The usefulness of this concept for our purposes is expressed by the fact that one again has a 1--1 correspondence between line bundles and generalised divisors on complex curves $X'$. By virtue of this fact, we will define the spectral data corresponding to a rank 1 commutative algebra $A$ (essentially) as the pair $(X',\Sss')$ comprising the eigenvalue curve $X'$ and the generalised divisor $\Sss'$ on $X'$ that corresponds to the eigenline bundle $\Lambda'$ on $X'$. It is also one of the purposes of this paper to explore the extent of the usefulness of generalised divisors in Hartshorne's sense, also for non-Gorenstein curves, and to convince the reader of their manifest value. 

We do not consider the case of commutative algebras of rank higher than $1$ (which are generated by differential operators whose degrees are not relatively prime), because in this case the eigenvector bundle of $M(\lambda)$ is a vector bundle of rank higher than $1$. Such vector bundles do not correspond to generalised divisors, because the sheaf of their sections is not contained in the sheaf of meromorphic functions on $X'$. For this reason the construction and investigation of spectral data in this case would require a very different theory. 

In Section~\ref{Se:inverse} we will solve the inverse problem for the spectral data thus defined for rank 1 algebras $A$ of commuting differential operators. This means that we will prove that $A$ is essentially uniquely determined by its spectral data (in fact, the domain of definition of the differential operators is extended to a certain maximum), and we will also see how to reconstruct $A$ from its spectral data (see Theorems~\ref{T:inverseproblem} and \ref{T:inverseproblemsolution}). This constitutes the rephrasing of Burchnall's and Chaundy's main classification result in terms of the present, modern concepts. 

We will additionally show that two rank 1 algebras of commuting differential operators are isomorphic to each other as algebras if and only if the corresponding spectral curves $X'$ are biholomorphic (see Theorem~\ref{P:isomorphic-algebras}). In other words, the family of all rank 1 algebras of commutative differential operators that are isomorphic to a given one $A$ can be generated by taking the spectral data $(X',\mathcal{S}')$\, corresponding to \,$A$\, and then varying the spectral divisor \,$\mathcal{S}'$\, throughout its connected component in the space of generalised divisors on \,$X'$\,. To the best of our knowledge, this statement is new in the sense that it has no counterpart in Burchnall's and Chaundy's classical work. 

We know from the discussion in \cite[Section~4]{KLSS} that the pairs \,$(X',\mathcal{S}')$\, occur in families where the complex curves $X'$ are partial normalisations, i.e.~branched one-fold coverings, of one another, and the generalised divisors \,$\mathcal{S}'$\, are direct images under the corresponding covering maps. These spectral data obtained by partial normalisation correspond to commutative algebras which contain $A$ as subalgebra. Such families always contain one member \,$(X'',\mathcal{S}'')$\, of minimal \,$\delta$-invariant, i.e.~minimal singularity. $X''$ was called the $\mathcal{S}'$-halfway normalisation of $X'$ in \cite{KLSS}. This minimal member corresponds to the maximal commutative rank 1 algebra which contains $A$, i.e.~to the centraliser of $A$ in the algebra of all ordinary differential operators. This construction also permits to find pairs \,$(X''',\mathcal{S}''')$\, which are below \,$(X',\mathcal{S}')$\, by a branched one-fold covering and so that \,$X'''$\, has arbitrarily large \,$\delta$-invariant. Such pairs correspond to rank 1 subalgebras of the given commutative rank 1 algebra $A$. 

We now begin our work by deriving a certain standard form for pairs of commuting differential operators in Section~\ref{Se:algebra}, which will facilitate the construction of the spectral data. 

\section{The algebra of differential operators}
\label{Se:algebra}

Let us first introduce the algebra of ordinary differential operators. We consider three different algebras of differential operators:
\begin{enumerate}
\item[Case 1:] The domain is an open interval $I=(a,b)$ and the algebra is $\Ass(I):=\C^{\infty}(I,\R)[D]$ where $D=\frac{d}{dt}$ with parameter $t\in I$.
\item[Case 2:] The domain $I$ is a real $1$-dimensional non-compact submanifold of $\C$ which is simply connected and the algebra is $\Ass(I):=C^{\infty}(I,\C)[D]$ where $D=\frac{d}{dz}$ with parameter $z\in I$.
\item[Case 3:] The domain is an open, connected subset $I$ of $\C$ and the algebra is $\Ass(I):=\Oss_I[D]$ where $D=\frac{d}{dz}$ with parameter $z\in I$.
\end{enumerate} 
Along with the domain \,$I$\, we fix a point \,$t_0 \in I$\, in any of these cases. 

We write these polynomials as
\[P=\alpha_mD^m+\ldots+\alpha_0.\] 
We denote the operator of multiplication with a function $f$ also by $f$. The algebra results as a subalgebra of the linear endomorphisms of the coefficient functions defined on $I$.
Due to the Leibniz rule the commutator of the differential operator $D$ and the operator of multiplication with a coefficient function $f$ acting on  functions on $I$ is given by the multiplication with $f'$, the derivative of $f$. Therefore, we set $Df-fD=f'$.
\begin{Lemma} \label{lem:ordinary differentials}
For any coefficient function $f$ on $I$ and any $n\in{\mathbb N}$ the composition of the operator $D^n$ with  $f$ acting on the functions on $I$ satisfies the identity
$$D^nf=\sum\limits_{0 \leq i\leq n}\binom{n}{i}f^{(i)}D^{(n-i)}.$$
\end{Lemma}

\begin{proof}
Let $\ad(D)$ denote the operator $A\mapsto[D,A]$ acting on the linear operators on the space of smooth functions. Then the Leibniz rule may be written as $\ad(D)f=f'$. Hence we have the identity $Df=fD+\ad(D)f=(\ad^0(D)f)D^1+(\ad^1(D)f)D^0$. This implies
$$D^nf=\sum\limits_{0 \leq i\leq n}\binom{n}{i}
\left(\ad ^{i}(D)f\right)D^{(n-i)}=
\sum\limits_{0 \leq i\leq n}\binom{n}{i}f^{(i)}D^{(n-i)}.$$
\end{proof}

If $P$ and $Q$ are two elements of $\Ass(I)$, then due to this formula the product $PQ$ is again an element of $\Ass(I)$. Since this product was derived from the action of the ordinary differential operators on the smooth functions, it endows $\Ass(I)$ with the structure of an associative algebra (a subalgebra of the operators on the smooth functions on $I$). The order of an element $P$ of $\Ass(I)$ is the highest number $m$ whose coefficient $\alpha_m$ does not vanish identically on $I$. Let us first use two transformations in order to bring a commutative subalgebra into standard form.
The first transformation changes the domain $I$ of the corresponding coefficient functions of $\Ass(I)$.
If $\xi$ is a smooth resp.~holomorphic, invertible function on $I$, the second transformation $P\mapsto\xi^{-1}P\xi$ is an inner automorphism of $\Ass(I)$. Now these two transformations may be used in order to bring a commuting pair of differential operators into standard form:
\begin{Proposition}\label{prop:standard form}
Let $P\in \Ass(I)$ be a differential operator of order $m$. Then $P$ can be transformed into an element $\widetilde{P}\in \Ass(\tilde{I})$ with highest coefficient equal to $1$ and  vanishing second highest coefficient. Thereby all operators $Q\in \Ass(I)$ which commute with $P$ are transformed into $\widetilde{Q}\in \Ass(\tilde{I})$ which commute with $\widetilde{P}$ and have constant highest and second highest coefficient.
%
%
\end{Proposition}

\begin{proof}
We first consider case 1. 
 If $\xi(t)$ is a diffeomorphism of an open interval $I$ onto another open interval $\tilde{I}$, the vector field $\frac{d}{dt}$ is transformed under this diffeomorphism onto the vector field $1/\xi'(t) \frac{d}{d\xi}$. This transformation therefore induces an isomorphism of the algebras $\Ass(I)$ and $\Ass(\tilde{I})$. We now construct such a deformation so that the highest coefficients of $\widetilde{P}$ and ultimately $\widetilde{Q}$ are equal to $1$.  

The highest coefficient $\alpha_m$ of $P$ cannot vanish identically. So there exists a subinterval of $I$ on which $\alpha_m$ has no roots. Then there exists an invertible function $\chi$ such that $\alpha_m=\chi^m$.  Hence, $\frac{d\xi}{dt}=\chi^{-1}$ 
and $\xi = \int \chi^{-1}dt$ is strictly monotonous and therefore a diffeomorphism of some subinterval of $I$ with some interval $\tilde{I}$ with the desired properties. Then $P\in\Ass(\tilde{I})$ has highest coefficient one. We define $\eta:=\frac{\alpha_{m-1}}{m}$. Then the inner automorphism corresponding to the function $\zeta=\exp\left(-\int\eta \,dt\right)$ transforms $P$ into a differential operator of the desired form.
If an operator  $Q\in\Ass(I)$ of order $n$  commutes with $P$, then the above transformation maps $Q$ to an operator $\widetilde Q\in\Ass(\tilde I)$. Since this transformation is an algebra homomorphism, $\widetilde{Q}$ commutes with $\widetilde{P}$.
Lemma \ref{lem:ordinary differentials} yields that the coefficient of $D^{m+n-1}$ in $\widetilde Q\widetilde P-\widetilde P\widetilde Q$ is equal to
\begin{align}\label{eq:highest coefficient}
n\widetilde \beta_n\widetilde \alpha_m'-m\widetilde \alpha_m\widetilde \beta_n'=0,
\end{align}
where $\widetilde{\alpha}_m$ and $\widetilde{\beta}_n$ are the highest coefficients of $\widetilde P$ and $\widetilde Q$, respectively. Therefore, $\widetilde\beta_n$ is constant if $\widetilde\alpha_m$ is constant.
The coefficient of $D^{m+n-2}$ in $\widetilde Q\widetilde P-\widetilde P\widetilde Q$ is due to Lemma \ref{lem:ordinary differentials} proportional to 
$$n\widetilde\alpha_{m-1}'-m\widetilde\beta_{n-1}'.$$
Therefore, $\widetilde\beta_{n-1}$ is constant if $\widetilde\alpha_{m-1}$ is constant.

In case 2, the proof is essentially the same as in case 1 with all (sub)intervals replaced by (sub)manifolds of $\C$. We now choose $\xi=\int \eta^{-1} dz$ which is complex-valued. Due to the inverse function theorem, $\xi$ defines a diffeomorphism on a possibly smaller submanifold of $I$ onto another submanifold $\tilde{I}$.

In case 3, there are no big changes from the proof of case 1 either. The (sub)intervals are replaced by open subsets of $\C$, the smooth maps are replaced by holomorphic maps and the diffeomorphisms by biholomorphic maps. Again the inverse function theorem gives that $\xi$ is a biholomorphic function on a possibly smaller open subset of $I$.
\end{proof}

We remark that if $m$ is equal to one, an inductive application of Proposition \ref{prop:standard form} shows that the coefficients of $Q$ are constant and $Q$ is a polynomial with respect to $P$.  
In the sequel, we shall only consider elements of a commutative subalgebra of $\Ass(I)$.  

\begin{Definition}
A commutative subalgebra of $\mathcal{A}(I)$ is called of rank $1$ if it contains two differential operators of coprime orders. We denote the set of such subalgebras by $\mathcal{R}$ in the following. 
\end{Definition}
In the sequel, we consider only commutative subalgebras of rank $1$.

\begin{Lemma}
\label{L:d0}
A commutative subalgebra $A$ of $\Ass(I)$ is of rank 1 if and only if there exists $d_0\in \mathbb{N}$ so that $A$ contains elements of every degree $d \geq d_0$.
\end{Lemma}

\begin{proof}
For $A \in \Rss$, $A$ contains a pair $(P,Q)$ of differential operators of coprime orders $n,m$. By B\'{e}zout's identity, there exist integers $a,b$ with $1=an+bm$. If $n=1$ or $m=1$, then $A$ contains operators of every positive order, so we now suppose $n,m \geq 2$. Then $ab<0$, and we suppose without loss of generality that $a>0$ and $b<0$. For $d \geq nm$ there exists an integer $l$ so that $dan-d \leq lmn \leq dan$, and then $d = (da-lm)n + (db+ln)m$ with $da-lm \geq 0$ and  $db+ln=\tfrac{d-dan+lnm}{m} \geq 0$. Then $P^{da-lm}\cdot Q^{db+ln} \in A$ is of order $d$. 

The converse follows because for every $d_0 \in \N$ there exist two coprime numbers \,$n,m \geq d_0$.
\end{proof}

The observation that the highest coefficients of all elements of $A\in\mathcal{R}$ are constant allows to define the degrees of the elements intrinsically.

\begin{Lemma}
\label{L:deg}
For $A\in\mathcal{R}$ and $P\in A\setminus\{0\}$ the degree $\deg(P)$ is equal to $\dim(A/PA)$ with $PA=\{PQ\mid Q\in A\}$
\end{Lemma}

\begin{proof}
By Lemma~\ref{L:d0} for $A\in\mathcal{R}$ the complement of the set
$$N=\{\deg(P)\mid P\in A\}$$
in $\mathbb{N}_0$ is finite. Since the highest coefficients of the elements of $A$ are constant, for $P\in A\setminus\{0\}$ the map $A\to A$ with $Q\mapsto PQ$ is injective. Therefore $\{\deg(PQ)\mid Q\in A\}$ is equal to $N+\deg(P)=\{d+\deg(P)\mid d\in N\}$. Furthermore, $\dim(A/PA)=\#(N\setminus(N+\deg(P))$. The set $N+\deg(P)$ is a subset of $N$, since $PA$ is a subspace of $A$. Therefore $N\setminus(N+\deg(P))$ is equal to $(\mathbb{N}_0\setminus(N+\deg(P)))\setminus(\mathbb{N}_0\setminus N)$ and the cardinality of both sets are equal to $\#(\mathbb{N}_0\setminus(N+\deg(P)))-\#(\mathbb{N}_0\setminus N)$. This difference equals $\deg(P)$, since $\mathbb{N}_0\setminus(N+\deg(P))$ is the disjoint union $\{0,\ldots,\deg(P)-1\}\dot{\cup}((\mathbb{N}_0\setminus N)+\deg(P))$.
\end{proof}

We will always assume that differential operators are in standard form as in Proposition~\ref{prop:standard form}, which means that highest and second highest coefficients are constant. In constructing this standard form, we possibly made the domain \,$I$\, smaller. Our solution of the inverse problem (Theorem~\ref{T:inverseproblemsolution}(1)) will show that such commutative algebras are already uniquely determined by their restriction to arbitrarily small, open subsets of \,$I$\,. 

\section{The direct problem}
\label{Se:direct}

The following Theorem, together with some kind of converse, has already been shown in \cite{BC1}. We give a proof here with methods which better match the point of view we want to take in this work. 
\begin{Theorem}\label{th:commuting operators}
For two commuting ordinary differential operators $P$ and $Q$ of orders $m$ respectively $n$,  there exists a polynomial $f(\lambda,\mu)$ with constant coefficients with the following properties:
\begin{enumerate}
\item[(i)] If we associate to $\lambda$ the degree $m$ and to $\mu$ the degree $n$, then the common degree of $f(\lambda,\mu)$ is equal to $mn$. Moreover, the highest coefficient is equal to $\mu^m+c\lambda^n$ with some non-zero constant $c$.
\item[(ii)] The differential operator $f(P,Q)$ is identically equal to zero.
\end{enumerate}

\end{Theorem}

\begin{proof}
In this proof we index the coefficients of \,$P$\, and \,$Q$\, differently from Section~\ref{Se:algebra} to simplify the characterisation of the degree of the coefficients of \,$f$\,. Specifically we write \,$P=D^m+\alpha_{1}D^{m-1}+\ldots+\alpha_m$\, and \,$Q=\beta_0D^n + \beta_1D^{n-1}+\ldots+\beta_n$\,. 

 For $\lambda\in\C$, we collect the derivatives
$\psi,\psi',\ldots,\psi^{(m-1)}$ of the solutions of the differential equation $(\lambda-P)\psi=0$ to a column-vector-valued function $\widehat\psi=(\widehat\psi_0,\dotsc,\widehat\psi_{m-1})^T$. Consequently, the differential equation $(\lambda-P)\psi=0$ is equivalent to the first order differential equation $\left(D-U(\,\cdot\,,\lambda)\right)\widehat\psi=0$ with the $m\times m$-matrix-valued function
\begin{equation}
\label{eq:def-U}
U(t,\lambda)=\begin{pmatrix}
0 & 1 & 0 & \ldots & 0\\
0 & 0 & 1 & \ldots & 0\\
\vdots & \vdots & \vdots & \ldots & \vdots\\
\lambda-\alpha_m(t) & -\alpha_{m-1}(t) & -\alpha_{m-2}(t) & \ldots & -\alpha_{1}(t)
\end{pmatrix}.
\end{equation}
With the help of the equation $P\psi = \lambda\psi$ we may express all derivatives of $\psi$ of order higher than $m-1$ in terms of the components of $\widehat\psi$:
That equation is equivalent to 
\[
\psi^{(m)} = \lambda\psi- \sum_{l=0}^{m-1}\alpha_{m-l} \psi^{(l)}  
= \lambda\hat\psi_0-\sum_{l=0}^{m-1}\alpha_{m-l} \hat\psi_{l}
\]
Differentiating this formula yields
\begin{align*}
\psi^{(m+1)} &= \lambda\psi'- \sum_{l=0}^{m-1}\alpha_{m-l} \psi^{(l+1)}-\sum_{l=0}^{m-1}\alpha'_{m-l} \psi^{(l)}    \\
&=\lambda\hat\psi_1\!-\!(\alpha_{m}'\!+\!\alpha_{1}\lambda\!-\!\alpha_{1}\alpha_m)\widehat\psi_0\!+\!\sum_{l=1}^{m-1}(\alpha_{1}\alpha_{m-l}\!-\!\alpha_{m-l+1}\!-\!\alpha'_{m-l}) \hat\psi_l. 
\end{align*}
The higher derivatives of $\psi$ can be obtained by inductively repeating this procedure. Note that the sum of the indices on the right hand side in each term plus the order of the derivative in this term always equals the order of the derivative on the left hand side. 

In particular, there exists a unique $m\times m$-matrix $V(t,\lambda)$ whose coefficients are polynomials in $\lambda$ and differential polynomials of the coefficients $\alpha_i$ of $P$ and $\beta_j$ of $Q$ such that $V(\,\cdot\,,\lambda)\widehat\psi=Q\widehat\psi$ for all $\psi$ in the kernel of $\lambda-P$.

For any complex number $\lambda\in{\mathbb C}$ we consider the solutions of the differential equation $(P-\lambda)\psi=0$. Due to the theory of ordinary differential equations, there exist exactly $m$ linear independent solutions $\psi_1,\ldots,\psi_m$ of this ordinary differential equation on $I$. Two solutions coincide if and only if the corresponding values of $\psi,\psi',\ldots,\psi^{(m-1)}$ at any element of the domain coincide. 
In particular, a base $\psi_1,\ldots,\psi_m$ of all solutions is uniquely determined by the condition that the derivatives up to order $m-1$ of $\psi_i$ vanish at the marked point  $t_0$ with the exception of the $(i-1)$-th derivative, which is equal to one at $t_0$. Since $Q$ commutes with $P$, the span of $\psi_1,\dots,\psi_m$ is invariant with respect to $Q$. Hence $Q$ acts as right multiplication with an $m\times m$-matrix $M(\lambda)$ on the row vector $(\psi_1,\dots,\psi_m)$. 

%
The vectors $\widehat\psi$ corresponding to the basis $\psi_1,\ldots,\psi_m$ build the fundamental solution $g(t,\lambda)$, i.e.  an $m\times m$-matrix-valued function depending on $(t,\lambda)\in I\times{\mathbb C}$:
$$(D-U(\,\cdot\,,\lambda))g(\,\cdot\,,\lambda)=0\text{ and } g(t_0,\lambda)=\unity,$$
where $g(t,\lambda)$ is invertible for all $(t,\lambda)\in I\times{\mathbb C}$. 
We conclude 
\begin{equation}
\label{eq:M-and-V}
V(t,\lambda)g(t,\lambda)=g(t,\lambda)M(\lambda) \, \Leftrightarrow \, M(\lambda)=g^{-1}(t,\lambda)V(t,\lambda)g(t,\lambda)
\end{equation} 
where $M(\lambda)$
does not depend on $t$, but it does depend on the choice of the marked point \,$t_0$\,. Since $g(t,\lambda)$ is a solution of the differential equation $(D-U(\,\cdot\,,\lambda))g(\,\cdot\,,\lambda)=0$ one has
\begin{multline*}
0=Dg^{-1}(\,\cdot\,,\lambda)V(\,\cdot\,,\lambda)g(\,\cdot\,,\lambda)=
-g^{-1}(\,\cdot\,,\lambda)U(\,\cdot\,,\lambda)g(\,\cdot\,,\lambda)+\\
+g^{-1}(\,\cdot\,,\lambda)\frac{\partial V(\,\cdot\,,\lambda)}{\partial t}g(\,\cdot\,,\lambda)+
g^{-1}(\,\cdot\,,\lambda)V(\,\cdot\,,\lambda)U(\,\cdot\,,\lambda)g(\,\cdot\,,\lambda)=\\
g^{-1}(\,\cdot\,,\lambda)\left[D-U(\,\cdot\,,\lambda),V(\,\cdot\,,\lambda)\right]g(\,\cdot\,,\lambda).
\end{multline*}
Therefore, the commutativity of the operators $P$ and $Q$ implies
$$\left[D-U(\,\cdot\,,\lambda),V(\,\cdot\,,\lambda)\right]=0.$$ 
Since the characteristic polynomial of $V(t,\lambda)$ is invariant under conjugation of $V(t,\lambda)$ with $g(t,\lambda)$, this yields
$$f(\lambda,\mu):=\det\left(\mu\unity-V(t,\lambda)\right)=
\det\left(\mu\unity-M(\lambda)\right). $$
So $f(\lambda,\mu)$ does not depend on $t$. Due to our construction, $U(t,\lambda)$, $V(t,\lambda)$ and $M(\lambda)=V(t_0,\lambda)$ are polynomials with respect to $\lambda$. 

Let us determine the highest coefficients of these polynomials. 
In order to regard $P$ as homogeneous of degree $m$ and $Q$ as homogeneous of degree $n$, we assign the weight $i$ to $\alpha_i$ for $i=0,\dots,m$, the weight $j$ to $\beta_j$ for $j=0,\dots,n$ and to every derivative the weight $1$. This assignment is in accordance with the weight \,$m$\, for $\lambda$ and the weight \,$n$\, for $\mu$ as stated in the theorem. 
The $k$-th row of $V(\,\cdot\,,\lambda)$ describes the action of $Q$ on $\psi^{(k)}$. So the entry $V_{kl}$ of the matrix $V(\,\cdot\,,\lambda)$
has the weight $n+k-l$ since $V_{kl}\widehat\psi_l$ contains at most $n+k$ derivatives.
The entries $V_{kl}$ are therefore homogeneous polynomials of degree $n+k-l$ with respect to $\lambda$ and derivatives of $\alpha_i$ and $\beta_j$, where $\deg(\alpha_i^{(r)})=i+r$, $\deg(\beta_j^{(s)})=j+s$ and $\deg(\lambda)=m$. 
So $\det(V(t,\lambda))$ has the degree $mn$. 
Moreover the $(k,l)$-th entry of $\mu\unity-V(t,\lambda)$ has the weight $n+k-l$, so the characteristic polynomial of $V(t,\lambda)$ is homogeneous of weight $mn$.
In particular the highest coefficients of \,$f(\lambda,\mu)$\, depend only on the coefficients of \,$P$\, and \,$Q$\, of weight zero, which are the highest coefficients. 

Therefore the highest coefficients of \,$f(\lambda,\mu)$\, for general \,$P$\, and \,$Q$\, are already obtained by considering the ``free case'' $P=D^m$ and $Q=\beta_0\,D^n$ with a constant \,$\beta_0 \neq 0$\,. In this case, the matrices are 
\begin{align}\label{eq:UV}
U\ind{free}(t,\lambda)&=\begin{pmatrix}
0 & 1 & 0 & \ldots & 0\\
0 & 0 & 1 & \ldots & 0\\
\vdots & \vdots & \vdots & \ldots & \vdots\\
\lambda & 0 & 0 & \ldots & 0
\end{pmatrix},&V\ind{free}(t,\lambda)&=M\ind{free}(\lambda)=\beta_0\cdot U\ind{free}^n(t,\lambda)\;.
\end{align}   
Hence $\det(U\ind{free}(t,\lambda))=(-1)^{m-1}\lambda$ and the highest coefficient of $f(\lambda,\mu)$ is equal to $\mu^m-\beta_0^m\,\lambda^n$.  

 Now we claim that the differential operator $f(P,Q)$ vanishes identically. Due to the commutativity of $P$ and $Q$, this differential operator does not depend on the order in which the operators are inserted into the polynomial $f(P,Q)$. If $\psi$ is any common solution of the equations $(P-\lambda)\psi=0$ and $(Q-\mu)\psi=0$ with $f(\lambda,\mu)=0$, then $P$ acts on $\psi$ as the multiplication with $\lambda$ and $Q$ acts on $\psi$ as the multiplication with $\mu$. Consequently, the action of $f(P,Q)$ on $\psi$ is the same as the action of $f(\lambda,\mu)$ on $\psi$ and hence vanishes. For any roots $(\lambda_1,\mu_1),\ldots,(\lambda_r,\mu_r)$ in $\C^2$ of $f$ with pairwise different $\lambda_1,\ldots,\lambda_r$, any choice of non-trivial common solutions $\psi_1,\ldots,\psi_r$ of $(P-\lambda_k)\psi_k=0$ and $(Q-\mu_k)\psi_k=0$ for $k=1,\dots,r$ are linear independent. In fact, suppose that these solutions obey a linear relation
 $$a_1\psi_1+\ldots+a_r\psi_r=0.$$
 Then the action of $P,P^2,\ldots,P^{r-1}$ on the relation adds $r-1$ other linear relations
 $$a_1\lambda_1^s\psi_1+\ldots+a_r\lambda_r^s\psi_r=0\text{ for }s=1,\ldots,r-1.$$
 Altogether, we have $r$ linear relations on the functions $a_1\psi_1,\ldots,a_r\psi_r$. Since $\lambda_1,\ldots,\lambda_r$ are pairwise different, the determinant of coefficients of these relations is a non-vanishing Vandermonde determinant. This implies that all these functions $a_k\psi_k$ vanish identically. A solution $\psi$ of $(P-\lambda)\psi=0$ vanishes identically on $\tilde{I}$ if the first $m$ derivatives of $\psi$ vanish at $t_0\in \tilde{I}$. Hence, the complements of the sets of roots of $\psi_k$ are open and dense and there exists $t_0\in \tilde{I}$ with non-vanishing values $\psi_k(t_0)$. Therefore $a_1,\ldots,a_r$ vanish, and thus the $\psi_k$ are indeed linear independent. We conclude that the differential operator $f(P,Q)$ has infinitely many solutions $f(P,Q)\psi=0$. Since the order of $f(P,Q)$ is bounded by $nm$, this implies that this differential operator vanishes identically.
\end{proof}
A partial converse has been shown by Burchnall and Chaundy in \cite{BC1} which includes pairs of differential operators $P,Q$ of co-prime orders, but not the general case. In the sequel, we shall associate a singular curve to a commutative algebra of differential operators. 

From now on, we want to investigate pairs of commuting operators $P$ and $Q$ whose orders $m$ and $n$ are co-prime. We use the results and notation of \cite{KLSS}.
In the proof of Theorem \ref{th:commuting operators}, we constructed a holomorphic matrix-valued function $M: \C \to \C^{m\times m}$. In \cite[Section 4]{KLSS} we have described how this matrix can be associated with a pair $(X',\Sss')$ where $X'$ is a complex curve and $\Sss'$ a generalized divisor.

In Theorem \ref{th:commuting operators} a pair of commuting differential operators is given.  For such a pair,
we define the singular curve $X'$ as the one-point compactification of 
\begin{equation}
\label{eq:eigenvaluecurve}
\bigr\{ \, (\lambda,\mu) \in \C\times \C \ |\ \det\bigr(\mu\cdot \unity-M(\lambda)\bigr)=0 \, \bigr\} \; . 
\end{equation}
\,$X'$\, does not depend on the choice of the marked point \,$t_0$\, because of Equation~\eqref{eq:M-and-V}. 
 
We claim that $X'$ is a singular curve and the point at infinity $\infty$ a smooth point.
Since $P$ and $Q$ are in standard form with highest coefficient equal to $1$, $M(\lambda)$ has the following form:
\begin{equation}
\label{eq:asymp-M}
M(\lambda)=\begin{pmatrix}
0 & 1 & 0 & \ldots & 0\\
0 & 0 & 1 & \ldots & 0\\
\vdots & \vdots & \vdots & \ldots & \vdots\\
\lambda & 0 & 0 & \ldots & 0
\end{pmatrix}^n+O(\lambda^{n-1}), \text{ as }\lambda\to\infty.
\end{equation}
Therefore, there exists a local parameter $z$ defined for large $|\lambda|$ such that $\lambda=(z/2\pi\ci)^{-m}$, $\mu=(z/2\pi\ci)^{-n}+O(z^{1-n})$ as $z\to 0$. If \,$m$\, and \,$n$\, are relatively prime, then \,$z$\, is uniquely characterised by these conditions.  
The single root of $z$, which is added at infinity to \eqref{eq:eigenvaluecurve}, is a smooth point $\infty$ of $X'$. The eigenvalue $\lambda$ corresponds to a meromorphic function on $X'$ with a single pole of order $m$ at $\infty$ and $\mu$ to a meromorphic function with a single pole of order $n$ at $\infty$.
In order to define the generalized divisor $\Sss'$ we normalize the eigenfunction of $M(\lambda)$ by $\ell(\psi)=1$ with $\ell: \C^m \to \C,\ (\psi_{1},\dots,\psi_{m})\mapsto \psi_{1}$. 
\begin{Lemma}
\label{L:normalize-bounded}
There exist only finitely many $(\lambda,\mu) \in X'\setminus \{\infty\}$ for which the kernel of $\ell$ contains non-trivial eigenvectors of $M(\lambda)$ with eigenvalue $\mu$. 
\end{Lemma}
\begin{proof}
We consider $V(t,\lambda)$ as defined in the proof of Theorem \ref{th:commuting operators} and $V\ind{free}(t,\lambda)$ as defined in equation \eqref{eq:UV} which corresponds to the free case $P=D^m$ and $Q=D^n$. The normalized free eigenvector $\varphi$ obeys
\begin{equation}\label{eq:free_V}
V\ind{free}(t,\lambda)\varphi=\lambda^{n/m}\varphi \mbox{ with }
\varphi=\left(1,\lambda^{1/m},\dots, \lambda^{(m-1)/m} \right)^T.
\end{equation} 
Due to the weights of $V_{kl}$, as introduced in the proof of Theorem \ref{th:commuting operators}, $V_{kl}$ is a polynomial of degree $n+k-l$. Therefore, all contributions to $V_{kl}$ which do not contribute to $V_{\mathrm{free},kl}$ include at least one of the coefficients $\alpha_i$ for $i\leq m-2$ or $\beta_j$ for $j\leq n-1$, or a derivative of such a coefficient. This implies
\[
|(V_{kl}-V_{\mathrm{free},kl})|=\mathrm{O}\left(\lambda^{-1/m}\right)\lambda^{(n+k-l)/m}.
\] 
Let $T$ be the diagonal matrix $\mathrm{diag}\left(1,\lambda^{-1/m},\dots, \lambda^{-(m-1)/m}\right)$. Therefore, 
\[
|(T(V-V\ind{free})T^{-1})_{kl}|=|(V-V\ind{free})_{kl} \lambda^{(l-k)/m}|=|\lambda|^{n/m}O(
\lambda^{-1/m}),
\]
Due to equation \eqref{eq:UV}
\[
\lambda^{-n/m}\cdot TV\ind{free}T^{-1}= \begin{pmatrix}
0 & 1 & 0 & \ldots & 0\\
0 & 0 & 1 & \ldots & 0\\
\vdots & \vdots & \vdots & \ldots & \vdots\\
1 & 0 & 0 & \ldots & 0
\end{pmatrix}^n
\] 
and hence
\[
|\lambda|^{-n/m}\cdot\|TVT^{-1}- TV\ind{free}T^{-1}\| \leq O(\lambda^{-1/m}).
\]
Since $TV\ind{free}T^{-1}$ has $m$ pairwise different eigenvalues, it is diagonalisable. Now we show that the distance $\|T\psi-T\phi\|$ of the eigenfunctions is also of order $O(\lambda^{-1/m})$. Due to the implicit function theorem applied to $(TVT^{-1}-\mu\unity)T\varphi$, the normalized eigenfunction $T\varphi$ and the eigenvalue $\mu$ depend nearby $TV\ind{free}T^{-1}$ continuously differentiably on the entries of $TVT^{-1}$.
Since $\ell \circ T=\ell$ and $\varphi\not\in\ker(\ell)$, for sufficiently large $\lambda$ also $\psi\not\in \ker (\ell)$. This shows that the set of \,$(\lambda,\mu) \in X' \setminus \{\infty\}$\, so that \,$\ker(\ell)$\, contains a non-trivial eigenvector of \,$M(\lambda)$\, for the eigenvalue \,$\mu$\, is a subvariety of \,$X'\setminus\{\infty\}$\, of codimension at least \,$1$\,, and hence finite.
\end{proof}

By evaluating \,$\psi$\, at the marked point \,$t=t_0$\,, we obtain 
the global meromorphic function \,$\chi := \psi(\,\cdot\,,t_0) = (\chi_1,\ldots,\chi_m)^T: X'\setminus\{\infty\} \to \C^m$. \,$\chi$\, is characterised uniquely by 
$$ M \chi = \mu \chi \quad \text{and} \quad 
   \ell(\chi) = 1,
$$ 
where we regard also $M$ and $\mu$ as functions on $X'\setminus\{\infty\}$. Locally, $\chi$ can be obtained from any holomorphic eigenfunction $\tilde{\chi}$ by taking $\chi=\tilde\chi/\tilde\chi_1$. 
  
In the sequel we use generalised divisors on the spectral curve \,$X'$\,. For this purpose we again apply the notations introduced in \cite{KLSS}. 
We define  
the generalised divisor $\Sss'$ corresponding to \,$\chi$\, on $X'\setminus \{\infty\}$ as the subsheaf of the sheaf 
of meromorphic functions on $X'\setminus \{\infty\}$ which is generated over $\Oss_{X'\setminus \{\infty\}}$ by $\chi_1,\dotsc,\chi_m$. Because of Lemma~\ref{L:normalize-bounded}, $\Sss'$ is equal to $\Oss_{X'}$ on a punctured neighborhood of $\infty$. We therefore extend $\Sss'$ to $\infty$ by defining $\Sss'_{\infty}=\Oss_{X',\infty}$.  

In Section~\ref{Se:inverse} we will describe the dependence of \,$\mathcal{S}'$\, on the marked point \,$t_0$\, by means of the Krichever construction. 

Note that the eigenspace of $M(\lambda)$ with eigenvalue $\mu$ is one-dimensional at all points of $X'$ where the map $X'\to \mathbb{P}^1,\, (\lambda,\mu)\mapsto \lambda$ is not ramified. If the eigenspaces of $M$ define a line bundle on $X'$, then $\Sss'$ describes the dual eigenline bundle in the sense of the correspondence between divisors and line bundles, see \cite[\S 29]{Fo}. The above considerations can be summarized in the assignment of a pair $(X',\Sss')$ to the pair $(P,Q)$ of commuting differential operators of co-prime orders.

In \cite[Definition 4.2]{KLSS} we have introduced the \emph{\,$\Sss'$-halfway normalisation} $X(\Sss')$ for the pair $(X',\Sss')$: For a generalised divisor $\Sss'$ on a singular curve $X'$, the \,$\Sss'$-halfway normalisation of \,$X'$\, is the unique one-sheeted covering $\pi_{X(\Sss')}: X(\Sss') \to X'$ such that 
\[ (\pi_{X(\Sss')})_\ast \Oss_{X(\Sss')} = \{ f \in \Bar{\Oss}_{X'} \,|\, f\cdot g \in \Sss' \text{ for all $g\in \Sss'$} \} \; . \]
On $X(\Sss')$, there exists a unique generalized divisor $\Sss(\Sss')$ whose direct image with respect to $\pi_{X(\Sss')}$ equals $\Sss'$. For any $q\in X'$, we choose local generators $\phi_1,\dots,\phi_m$ of $\Sss'_q$. Then for any $q'\in \pi_{X(\Sss')}^{-1}[\{q\}]$,  $\Sss(\Sss')_{q'}$ is the $X(\Sss')$-submodule of $\Mss_{q'}$ generated by $\phi_1\circ\pi_{X(\Sss)}, \dots, \phi_m\circ \pi_{X(\Sss)}$, see \cite{KLSS}.


Until now, we have assigned the quadruple $(X',\Sss',\infty,z)$ to commuting differential operators $P$ and $Q$ of coprime orders $m$ and $n$. Here, $P$ and $Q$ correspond to meromorphic functions $\lambda$ and $\mu$ on $X'$ with poles of orders $m$ and $n$ only at $\infty$. The meromorphic functions on $X'$ having only poles at $\infty$ are equal to the algebra $\C[\lambda,\mu]/(f)$ with $f$ defined in Theorem \ref{th:commuting operators}. In the sequel, we will assign such quadruples $(X',\Sss',\infty,z)$ to commutative algebras. 
This assignment has the property that the commutative algebra is isomorphic to the algebra of meromorphic functions on $X'$ which have poles only at $\infty$. The quadruple $(X',\Sss',\infty,z)$ constructed above is assigned to the commutative algebra generated by the two differential operators $P$ and $Q$. 
Our main result gives an essentially $1-1$ correspondence between the commutative algebras in the following class and quadruples $(X',\Sss',\infty,z)$ of a compact singular curve $X'$ with smooth marked point $\infty$ and coordinate \,$z$\, near \,$\infty$\, and a generalized divisor $\Sss'$ on $X'$ whose degree is equal to the arithmetic genus of $X'$. In this section, we investigate the map from the algebra to the triple and in the following section the inverse of this map.

We will see that the following definition describes the maximal commutative subalgebras of $\mathcal{A}(I)$ which are of rank $1$. 
\begin{Definition}[centraliser]
For each subalgebra $A$ of $\mathcal{A}(I)$ the algebra
\begin{equation*}
C(A):=\{P\in \mathcal{A}(I)\ \mid \ \forall\, Q\in A \,: \, [P,Q]=0\}
\end{equation*}
is called the \emph{centraliser} of $A$.
\end{Definition}

\begin{Lemma}\label{lem:CA}
For each $A\in\mathcal{R}$, we have $C(A)\in\mathcal{R}$, and $A$ has finite codimension in $C(A)$.
\end{Lemma}

\begin{proof}
Because we consider only differential algebras where the highest order coefficients of the member operators are constant, Lemma~\ref{L:d0} implies $\dim(B/A)\leq d_0 <\infty$ for any \,$B\in\mathcal{R}$\, with \,$B \supset A$\,.
Therefore it suffices to prove that $C(A)$ is commutative because
$C(A)$ contains $A$. 

Let $A\in\mathcal{R}$ and $P,Q\in A$ be two differential operators of coprime orders $m$ and $n$. For $\lambda \in \C$, let $V_\lambda$ be the $m$-dimensional space of solutions of $P\psi=\lambda\psi$. As in the proof of Theorem \ref{th:commuting operators}, let $M(\lambda)$ denote the endomorphism $V_\lambda\to V_\lambda$ induced by $Q$. Due to Theorem \ref{th:commuting operators}, $f(\lambda,\mu)$ has weighted degree $mn$ with highest term $\mu^m+c\lambda^n$ with $c\neq 0$. For coprime $m$ and $n$, the $m$-th roots of $\lambda^n$ are all pairwise different. Hence, for large $\lambda$ the $m$ solutions of $f(\lambda,\mu)=0$ are also pairwise different and
$M(\lambda)$ has $m$ pairwise different eigenvalues. Because the discriminant is holomorphic, the same holds for $\lambda$ in an open and dense subset of $\C$.

Now, let $R,S \in C(A)$. Since they commute with $P$ and $Q$, they define endomorphisms $B(\lambda)$ and $C(\lambda)$ of $V_\lambda$ commuting with $M(\lambda)$.
The commutator $[R,S]\in C(A)$ induces the endomorphism $[B(\lambda), C(\lambda)]$ of $V_\lambda$.
If $M(\lambda)$ has pairwise different eigenvalues, it is diagonal with respect to an appropriate basis. Since $B(\lambda)$ and $C(\lambda)$ commute with $M(\lambda)$, this basis also diagonalises $B(\lambda)$ and $C(\lambda)$. Therefore, $B(\lambda)$ and $C(\lambda)$ commute.
This implies that the vector spaces $V_\lambda$ belong to the kernel of $[R,S]$ if $M(\lambda)$ has pairwise different eigenvalues. By definition, $V_\lambda\cap V_{\lambda'}=\{0\}$ for $\lambda\neq \lambda'$. As in the proof of Theorem \ref{th:commuting operators}, $[R,S]$ has an infinite dimensional kernel.
Since $[R,S]$ has finite order, its kernel can only be infinite dimensional if $[R,S]=0$.
\end{proof}

\begin{Definition}
\emph{Spectral data} are a quadruple \,$(X',\mathcal{S}',\infty,z)$\,, where \,$X'$\, is a compact singular curve, \,$\mathcal{S}'$\, is a generalised divisor on \,$X'$\, whose degree is equal to the arithmetic genus of \,$X'$\,, \,$\infty$\, is a smooth point of \,$X'$\, and \,$z$\, is a local coordinate of \,$X'$\, near \,$\infty$\,. 
\end{Definition}

We show in the following theorem that $C(A)$ has spectral data \linebreak
$(X(\Sss'), \Sss(\Sss'), \infty,z)$.
This will lay the foundation to construct the spectral data assigned to general algebras $A\in\Rss$. In the sequel, \,$\mathcal{M}$\, denotes the sheaf of meromorphic functions on \,$X(\mathcal{S}')$\,. We omit the subscript \,${}_{X(\mathcal{S}')}$\, because \,$(\pi_{X(\mathcal{S}')})_*$\, is an isomorphism of sheaves from \,$\mathcal{M}_{X(\mathcal{S}')}$\, onto \,$\mathcal{M}_{X'}$\,. 

\begin{Theorem}\label{th:direct center}
	Let $A\in\mathcal{R}$ and $P,Q\in C(A)$ two differential operators of coprime orders and $(X',\Sss',\infty,z)$ be the spectral data corresponding to the subalgebra of $C(A)$ generated by $P$ and $Q$. Then the triple $(X(\Sss'),\Sss(\Sss'),\infty)$ and the value \,$\mathrm{d}z(\infty)$\, are independent of the choice of $P$ and $Q$, $C(A)$ is isomorphic to the algebra
\begin{align}\label{eq:B}
B&:=\{f\in H^0(X(\Sss'),\mathcal{M})\,\mid\,\forall\, p\in X(\Sss')\setminus\{\infty\}:\,f_p\in\mathcal{O}_{X(\Sss'),p}\}
\end{align}
and $(\pi_{X(\Sss')})_\ast\Sss(\Sss')=\Sss'$.
\end{Theorem}

We will see in Section~\ref{Se:inverse} that the solution of the inverse problem for given spectral data \,$(X',\mathcal{S}',\infty,z)$\, depends on \,$z$\, only in terms of the value of \,$\mathrm{d}z(\infty)$\,.

\begin{proof}
Let $P,Q\in A$ be the differential operators of coprime orders $m$ and $n$ with the corresponding matrix $M(\lambda)$. The subalgebra \,$\langle P,Q\rangle$\, of \,$\mathcal{A}(I)$\, generated by \,$P$\, and \,$Q$\, is commutative, hence \,$\mathcal{R} \ni \langle P,Q \rangle \subset A$\, and therefore \,$A \subset C(A) \subset C(\langle P,Q\rangle)$\,. By Lemma~\ref{lem:CA}, \,$C(\langle P,Q \rangle)$\, is commutative and therefore contained in \,$C(A)$\,. This implies that \,$C(\langle P,Q \rangle)=C(A)$\,.

We now show that $C(A)$ is isomorphic to the algebra $B$ in \eqref{eq:B}. Let $R\in C(A)$. Since $[R,P]=0$, there exists for each $\lambda\in\C$ a matrix $N(\lambda)\in\C^{m\times m}$ which describes the action of $R$ on the kernel of $P-\lambda\cdot\unity$.
Since $[R,Q]=0$, we have $[M(\lambda),N(\lambda)]=0$ for all $\lambda\in\C$. Therefore, for each $(\lambda,\mu)\in X', N(\lambda)$ acts on the kernel of $M(\lambda)-\mu\cdot\unity$. For those $(\lambda,\mu)$ with one-dimensional $\ker(M(\lambda)-\mu\cdot\unity), N(\lambda)$ acts as multiplication with a complex number $\nu$. Such $(\lambda,\mu)$ build an open and dense subset of $X'$. Since the entries of $N(\lambda)$ are meromorphic, $\nu$ extends to a meromorphic function on $X'$. For all $\lambda\in\C$, all entries of $N(\lambda)$ are bounded and therefore also the eigenvalue $\nu$ of the $N(\lambda)$. This implies $\nu\in B$.

Conversely, in \cite{KLSS}, it has been proven that $(\hat\lambda\circ\pi_{X(\Sss')})_\ast\mathcal{O}_{X(\Sss')}$ is isomorphic to the sheaf of holomorphic $n\times n$ matrices on $\mathbb P^1$ which commute with $M(\lambda)$. Here, $\hat\lambda$ is the map $\hat\lambda:X'\to\C$ such that $(\lambda,\mu)\mapsto\lambda$.

This algebra isomorphism extends to an isomorphism of $B$ with matrices $N(\lambda)$ whose entries are polynomials with respect to $\lambda$ and commute with $M(\lambda)$. These are the matrices which describe the action of the elements of $C(A)$ on the kernel of $P-\lambda\cdot\unity$. This correspondence is $1$-to-$1$.

Let $P',Q'\in C(A)$ be another pair of differential operators of coprime orders. Since \,$\langle P,Q \rangle \cap \langle P',Q' \rangle$\, contains all \,$R \in C(A)$\, of sufficiently large order by Lemma~\ref{L:d0}, there exists a third pair \,$P'',Q'' \in A$\, such that \,$\langle P'',Q'' \rangle \subset \langle P,Q \rangle \cap \langle P',Q' \rangle$\,. Without loss of generality, we may therefore suppose \,$P',Q' \in \langle P,Q \rangle$\,, meaning that \,$P'$\, and \,$Q'$\, can be regarded as polynomials in \,$P$\, and \,$Q$\,. 

The spectral data of $P',Q'$ is a quadruple $(X'',\Sss'',\infty',z')$ together with two meromorphic eigenfunctions $\lambda',\mu'$ on \,$X''$\,. \,$\lambda'$\, and \,$\mu'$\, can be regarded as polynomials in \,$\lambda$\, and \,$\mu$\,, and in this way we obtain a holomorphic map \,$X' \setminus \{\infty\} \to X'' \setminus \{\infty'\}$\,. Because \,$\infty$\, and \,$\infty'$\, are smooth points of \,$X'$\, and \,$X''$\,, respectively, this map extends to a holomorphic map \,$X' \to X''$\,, which is biholomorphic on an open and dense subset of \,$X'$\, by Lemma~\ref{L:normalize-bounded}. Therefore this map is a one-fold covering. The pullback of a common eigenfunction of \,$P'$\, and \,$Q'$\, is a common eigenfunction of \,$P$\, and \,$Q$\,, or equivalently, the direct image of \,$\mathcal{S}'$\, is \,$\mathcal{S}''$\,. By definition of \,$X(\mathcal{S}')$\,, it follows that \,$(X(\mathcal{S}''),\mathcal{S}(\mathcal{S}''),\infty)$\, is isomorphic to \,$(X(\mathcal{S}'),\mathcal{S}(\mathcal{S}'),\infty')$\,. Because \,$P$\, and \,$P'$\, have highest coefficient \,$1$\,, it follows from the definition of \,$z$\, that the corresponding biholomorphic map maps \,$\mathrm{d}z(\infty)$\, onto \,$\mathrm{d}z'(\infty')$\,.  
\end{proof}
\begin{Theorem}\label{th:direct general}
For $A\in\mathcal{R}$ and $P,Q\in A$ of coprime orders, let $(X',\Sss',\infty,z)$ be the corresponding spectral data. Then up to isomorphy, there exists a unique one sheeted covering $\pi'':X''\to X'$ and a generalized divisor $\Sss''$ on $X''$ with the following properties:
\begin{enumerate}[(i)]
\item $\pi_*''\Sss''=\Sss'$.
\item The following diagram commutes
$$
\begin{matrix}
\langle P,Q\rangle&\hookrightarrow&A&\hookrightarrow&C(A)\\
g_1\downarrow\cong&&g_2\downarrow\cong&&g_3\downarrow\cong\\
\frac{\C[\lambda,\mu]}{(f)}&\hookrightarrow&C&\hookrightarrow&B \;.
\end{matrix}
$$
Here, the polynomial $f$ and the isomorphism $g_1$ are defined in Theorem~\ref{th:commuting operators}, and $B$ and $g_3$ are defined in Theorem~\ref{th:direct center}. We also set
 \begin{align}\label{eq:C}
 C&:=\{f\in H^0(X'',\mathcal{M})\mid\forall p\in X''\setminus(\pi'')^{-1}[\{\infty\}]:\,f_p\in\mathcal{O}_{X'',p}\},
 \end{align} 
 and $g_2$ is defined by the above diagram.
\end{enumerate}
$(X'',\Sss'',\infty,z)$ does not depend on the choice of $P, Q \in A$.
\end{Theorem}
\begin{proof}
In Theorems~\ref{th:commuting operators} and~\ref{th:direct center}, we have shown that $g_1$ and $g_3$ are isomorphisms, respectively.
The sheaf $\Oss'_{X'}$ is contained in $\pi(\Sss')_{\ast}\Oss_{X(\Sss')}$. Note that $A$ is a subalgebra of $C(A)$. Hence the image of $A$ under $g_3$ is contained in $B$. We identify the meromorphic functions on $X(\Sss')$ with the meromorphic functions on $X'$, in this way $B$ becomes a subalgebra of $H^0(X'\setminus \{\infty\},\pi(\Sss')_{\ast}\Oss_{X(\Sss')})$. The image of $A$ in $B$ generates on $X'\setminus \{\infty\}$ a subsheaf $\Ass$ of subrings of  $\pi(\Sss')_{\ast}\Oss_{X(\Sss')}$. It contains $\Oss_{X'}$ since $A$ contains $P$ and $Q$. The stalks of the latter subsheaf have finite codimension in the stalks of $\pi(\Sss')_{\ast}\Oss_{X(\Sss')}$, and the codimension is $0$ away from the singularities of $X'$. Therefore we may extend $\Ass$ to $X'$ by $\Ass=\Oss_{X'}$ near the smooth point $\infty$. By definition of $X(\Sss')$, $\pi(\Sss')_{\ast}\Oss_{X(\Sss')}$ acts on $\Sss'$. Due to \cite[Lemma 4.1]{KLSS}, there exists a unique one-sheeted covering $\pi'':X''\to X'$ such that $\pi''_\ast (\Oss_{X''})=\Ass$.
The sequence of one-sheeted coverings $X(\Sss')\to X''\to X'$ induces the embeddings in the lower row of the diagram. The embeddings of the upper row are obvious. 

It remains to show that there exists an isomorphism $g_2$ as in the diagram. Because $g_3$ is an isomorphism, there exists a subalgebra of $C(A)$ which is mapped isomorphically onto $C$ by $g_3$. It suffices to show that this subalgebra equals $A$. 

On the one hand, since the sheaf of subrings $\pi''_\ast (\Oss_{X''})$ of $\pi(\Sss')_{\ast}\Oss_{X(\Sss')}$ is generated by the image of $A$ under $g_3$ in $B$, this algebra is contained in $A$. On the other hand, the image of every element of $A$ with respect to $g_3$ in $B$ belongs to the subalgebra which generates $\pi''_\ast (\Oss_{X''})$ and therefore to the image of $C$ in $B$. So this subalgebra contains $A$.  

The only choice that was made in this construction was that of \,$P$\, and \,$Q$\,. The independence of \,$(X'',\mathcal{S}'')$\, from the choice of \,$P$\, and \,$Q$\, follows by the same argument as in the proof of Theorem~\ref{th:direct center}. 
\end{proof}


\begin{Lemma}
The generalized divisors $\Sss'$, $\Sss(\Sss')$ and $\Sss''$ have degree equal to the arithmetic genus of $X'$, $X(\Sss')$, $X''$ respectively, and they are non-special.
\end{Lemma} 

\begin{proof}
We first show the claim for $X'$ and $\Sss'$. At $\infty$, the function $\psi_k$ has a pole of order $k-1$ because of \eqref{eq:free_V} and the asymptotics shown in the proof of Lemma~\ref{L:normalize-bounded}. Therefore every linear combination of the $\psi_k$ that is holomorphic at $\infty$ is a multiple of $\psi_1$. This shows that $\dim H^0(X',\Sss')=1$. Let $U' := \{(\lambda,\mu) \in X'\setminus \{\infty\}\bigr|\;|\lambda|>R\}$, where we choose $R>0$ large enough so that $\Sss'$ is equal to $\Oss_{X'}$ on $U'$. Then $\mathfrak{U} := (X'\setminus \{\infty\},U' \cup \{\infty\})$ is a Leray covering of $(X',\Sss')$ by \cite[Proposition~4.5]{KLSS}. We use this covering to show that $H^1(X',\Sss')=0$. Let $f\in H^1(\mathfrak{U},\Sss')$. By \cite[Proposition~4.5]{KLSS}, we have $f=f_1\psi_1 + \dotsc + f_m\psi_m$ with holomorphic functions $f_1,\dotsc,f_m$ on $U := \{\lambda \in \C\bigr|\;|\lambda|>R\}$. We can write $f_k=g_k-\lambda^{-1}\,h_k$, where $g_k$ is an entire function, and $h_k$ is a holomorphic function on $U \cup \{\infty\}$. $h := h_1\,\lambda^{-1}\,\psi_1 + \dotsc + h_m\,\lambda^{-1}\,\psi_m$ is holomorphic on $U'\cup\{\infty\}$, because the pole order of $\psi_k$ at $\infty$ is at most $m-1$, so $\lambda^{-1}\,\psi_k$ is holomorphic at $\infty$. With $g := g_1\,\psi_1 + \dotsc + g_m\,\psi_m$, we have $f=g-h$, and therefore $f$ is a boundary with respect to $(\mathfrak{U},\Sss')$. This shows that $H^1(X',\Sss')=H^1(\mathfrak{U},\Sss')=0$. By Riemann-Roch's Theorem \cite[Theorem~5.2]{KLSS}, it follows that $\deg(\Sss')$ equals the arithmetic genus of $X'$, and that $\Sss'$ is non-special.

We now consider $\Sss''$ on $X''$. On one hand, we have $H^0(X'',\Sss'')= H^0(X',\Sss')$ because of $\pi''_* \Sss''=\Sss'$. On the other hand, because of $\Sss' \supset \pi''_*\Oss_{X''}$, we have
\begin{align*}
\deg(\Sss') & = \dim H^0(X',\Sss'/\Oss_{X'}) \\
& = \dim H^0(X',\Sss'/\pi''_*\Oss_{X''}) + \dim H^0(X',\pi''_*\Oss_{X''}/\Oss_{X'}) \\ 
& = \deg(\Sss'') + (g(X')-g(X'')) \; . 
\end{align*}
Because $\deg(\Sss')$ equals the arithmetic genus $g(X')$ by the previous part of the proof, $\deg(\Sss'')=g(X'')$ follows. Therefore also $\Sss''$ is non-special. 

This argument likewise applies to $(X(\Sss'),\Sss(\Sss'))$.
\end{proof}

\begin{Proposition}
Suppose that we are in case~1, i.e.~\,$\mathcal{A}=C^\infty(I,\R)[D]$\, with an open interval \,$I \subset \R$\,. Then the spectral data \,$(X'',\mathcal{S}'',\infty,z)$\, of \,$A \in \mathcal{R}$\, satisfy the following reality conditions: 

\begin{enumerate}
\item 
There exists an anti-holomorphic involution \,$\rho$\, on \,$X''$\, so that \,$\infty$\, is a smooth point of the real singular curve given by the fixed point set of \,$\rho$\,, and \,$\rho^* \overline{z}=-z$\,. For any \,$P,Q \in A$\, of co-prime order, \,$\rho$\, acts on the eigenvalues \,$(\lambda,\mu)$\, as \,$(\lambda,\mu) \mapsto (\bar{\lambda},\bar{\mu})$\,. 

\item
We have \,$\rho_* \overline{\mathcal{S}''} = \mathcal{S}''$\,, where the generalised divisor \,$\rho_* \overline{\mathcal{S}''}$\, is characterised by 
$$ H^0(U,\rho_* \overline{\mathcal{S}''}) = \{ \overline{f} \circ \rho \mid f \in H^0(\rho(U),\mathcal{S}'') \} $$
for any open subset \,$U \subset X''$\,. 
\end{enumerate} 
\end{Proposition}

\begin{proof}
We consider differential operators \,$P,Q\in A$\, of co-prime order \,$m$\, and \,$n$\,, respectively. 
Because we are in case 1, \,$P$\, and \,$Q$\, have real coefficients. Therefore the matrices \,$U$\, and \,$V$\, from the proof of Theorem~\ref{th:commuting operators} are real for \,$\lambda \in \R$\, and thus also \,$M(\lambda)$\, is real for \,$\lambda \in \R$\,. This shows that \,$M(\bar{\lambda}) = \overline{M(\lambda)}$\, for all \,$\lambda \in \C$\,. Therefore the polynomial \,$f$\, has real coefficients, and hence \,$\rho: (\lambda,\mu) \mapsto (\bar{\lambda},\bar{\mu})$\, is an anti-holomorphic involution on the singular curve \,$X' \setminus \{\infty\}$\,.

We extend \,$\rho$\, to \,$X'$\, by setting \,$\rho(\infty)=\infty$\,. It was shown in Theorem~\ref{th:commuting operators} that the highest coefficient of the polynomial \,$f$\, is of the form \,$\mu^m+c\lambda^n$\, with a non-zero constant \,$c$\,, which is real in the present setting. Therefore \,$\infty$\, is a smooth point of the fixed point set of \,$\rho$\,, and \,$\rho^* \overline{z}=-z$\,. 

Because \,$X'' \to X'$\, is a one-sheeted covering, we obtain an anti-holomorphic map \,$\rho$\, on \,$X''$\, with the desired properties.

As the linear form \,$\ell$\, is real, the normalised section \,$\psi$\, also satisfies \,$\overline{\psi} \circ \rho = \psi$\,, and therefore \,$\rho_* \overline{\mathcal{S}''} = \mathcal{S}''$\, holds. 
\end{proof}
	
\section{The inverse problem}
\label{Se:inverse}

We now solve the corresponding inverse problem. We let spectral data \,$(X',\mathcal{S}',\infty,z)$\, be given. This means that \,$X'$\, is a singular curve with a marked smooth point \,$\infty$\, and a local parameter \,$z$\, defined on an open neighbourhood \,$U_1$\, of \,$\infty$\,, and \,$\mathcal{S}'$\, is a generalised divisor on \,$X'$\, of degree equal to the arithmetic genus of \,$X'$\,. 

We will use the Krichever construction as in \cite[Section~7]{KLSS}. In particular we define the one-parameter group of invertible sheaves \,$\mathcal{L}_{1/z}(t)$\, with \,$t\in \C$\,: Let \,$U_0 := X' \setminus \{\infty\}$\,, then \,$(U_0,U_1)$\, is a covering of \,$X'$\, and the cocycle \,$z^*\,\exp(-2\pi\ci t/z)$\, defines \,$\mathcal{L}_{1/z}(t)$\, with respect to this covering. On an open subset \,$O \subset \C$\,, the same cocycles with variable \,$t \in O$\, also define a sheaf \,$\mathcal{L}_{1/z}$\, on \,$X' \times O$\,.

The Krichever construction depends on the choice of the local parameter \,$z$\, only via the Mittag-Leffler distribution induced by \,$\tfrac{1}{z}$\,. 
For any two different local parameters \,$z_1,z_2$\, on \,$X'$\, around \,$\infty$\, there exists a constant \,$c = \tfrac{\mathrm{d}z_2}{\mathrm{d}z_1}(\infty)\neq 0$\, so that 
\,$\tfrac{1}{z_1}-\tfrac{c}{z_2}$\, is holomorphic. 
This shows that our construction in fact depends on the choice of the local coordinate \,$z$\, only in terms of the Taylor coefficient \,$\mathrm{d}z(\infty)$\,. 

As in \cite[Equation~(31)]{KLSS} we define
\begin{equation}
\label{eq:T-definition}
T := \{ t \in \C \mid H^0(X',\mathcal{S}'_{-\infty} \otimes \mathcal{L}_{1/z}(t)) \neq 0 \} \; ,
\end{equation}
where \,$\mathcal{S}'_{-\infty}$\, is the generalised divisor obtained by multiplying \,$\mathcal{S}'$\, with the invertible sheaf defined by the classical divisor \,$-\infty$\,.
In \cite[Theorem~8.6]{KLSS} it was shown that \,$T$\, is a subvariety of \,$\C$\,. For our specific situation we improve that result by the following lemma.

%

\begin{Lemma}
\label{L:Tdiscrete}
\,$T$\, is discrete.
\end{Lemma}

\begin{proof}
We assume on the contrary that \,$T$\, is not discrete. 
Because \,$T$\, is a subvariety of \,$\C$\, by \cite[Theorem~8.6]{KLSS}, this means that it contains an open subset \,$O_1 \subset \C$\,. 

Let \,$k>0$\, be the minimal dimension of \,$H^0(X',\mathcal{S}'_{-\infty} \otimes \mathcal{L}_{1/z}(t))$\, for \,$t \in O_1$\,. The sheaf \,$\mathcal{S}'_{-\infty}$\, on \,$X'$\, induces a sheaf on \,$X' \times O_1$\,, which we also denote by \,$\mathcal{S}'_{-\infty}$\,. Then the sheaf  \,$\mathcal{S}'_{-\infty}\otimes \mathcal{L}_{1/z}$\, on \,$X' \times O_1$\, is flat with respect to the projection \,$X' \times O_1 \to O_1$\, by \cite[Lemma~8.5]{KLSS}. Because of \cite[Chapter~III Theorem~4.7~(a)]{GPR}, the map \,$t \mapsto \dim(H^0(X',\mathcal{S}'_{-\infty}\otimes \mathcal{L}_{1/z}(t)))$\, is upper semi-continuous, and therefore the subset \,$O_2 \subset O_1$\, on which this dimension is equal to \,$k$\, is open. Due to \cite[Chapter~III Theorem~4.7~(d)]{GPR} the spaces $H^0(X',\mathcal{S}'_{-\infty} \otimes \mathcal{L}_{1/z}(t))$ are the fibres of a vector bundle over $t \in O_2$. In particular there exists a non-trivial section of $\mathcal{S}'_{-\infty}\otimes \mathcal{L}_{1/z}$ on $X'\times O_2$. By definition of \,$\mathcal{L}_{1/z}$\,, this section corresponds to a section \,$\psi$\, of \,$\mathcal{S}'_{-\infty}$\, on \,$U_0 \times O_2$\, such that the function 
\begin{equation}
\label{eq:L:discrete:phi}
\phi(x,t)=\psi(x,t)z\exp(-2\pi\ci\, t/z)
\end{equation}
is holomorphic on $U_1\times O_2$ and vanishes on \,$\{\infty\} \times O_2$\,. 

Let $\phi(x,t)=\sum_{n\ge 1}z^n\phi_n(t)$ be the Taylor expansion of $\phi$ with respect to the local coordinate $z$ at $\infty$. We let $N \geq 1$ be the smallest index such that $\phi_{N}$ does not vanish identically on \,$\{\infty\} \times O_2$\,. Then \,$O := \{ t\in O_2 \mid \phi_{N}(\infty,t)\neq 0\}$\, is an open subset of \,$O_2$\,.  Due to Equation~\eqref{eq:L:discrete:phi}, the $m$-th derivative $x \mapsto \frac{\partial^m}{\partial t^m}\psi(x,t)$ has for every $m\in\mathbb{N}$ and \,$t\in O$\, an $(m-N)$-th order pole at $x=\infty$. Furthermore, on $U_0$ these $m$-th derivatives are section of $\mathcal{S}'$. This implies that for all $m\in\mathbb{N}$ and \,$t\in O$\, the sheaf $\mathcal{S}'_{(m-N)\infty}\otimes \mathcal{L}_{1/z}(t)$ has a non-trivial section which does not belong to $\mathcal{S}'_{(m-N-1)\infty}\otimes \mathcal{L}_{1/z}(t)$. For sufficiently large $m$ the degree of $\mathcal{S}'_{(m-N)\infty}$ is greater than $2g-2$, and due to S\'{e}rre Duality~\cite[Corollary~6.6(c)]{KLSS} $H^1(X',\mathcal{S}'_{(m-N)\infty})$ is trivial. Now Riemann-Roch implies $\dim H^0(X',\mathcal{S}'_{(m-N)\infty})=m-N+1$. Because the derivatives $\frac{\partial^l}{\partial t^l}\psi(\infty,t)$ belong to this space for $l=0,\ldots,m$, we have $m-N+1\ge m+1$. This implies $N\le 0$, which contradicts \,$N\geq 1$\,.
\end{proof}

For every \,$t_0\in \C\setminus T$\,, \,$\mathcal{S}' \otimes \mathcal{L}_{1/z}(t_0)$\, is equivalent to a generalised divisor \,$\mathcal{S}''$\, with \,$\mathcal{O}_{X'} \subset \mathcal{S}''$\, and support contained in \,$X' \setminus \{\infty\}$\, by \cite[Lemma~8.4]{KLSS}. Because \,$\mathcal{S}'' \otimes \mathcal{L}_{1/z}(t)$\, is equivalent to \,$\mathcal{S}' \otimes \mathcal{L}_{1/z}(t_0+t)$\,, we then have
$$ \{ t \in \C \mid H^0(X',\mathcal{S}''_{-\infty} \otimes \mathcal{L}_{1/z}(t)) \neq 0 \} \;=\; t_0 + T  \; . $$
By \cite[Theorem~8.8]{KLSS}, \,$\mathcal{S}'$\, therefore induces a Baker-Akhiezer function \,$\psi: (X' \setminus \{\infty\}) \times (\C\setminus T) \to \C$\, such that the holomorphic extension of the function \,$\psi(x,t) \cdot \exp(-2\pi\ci t/z)$\, takes the value \,$1$\, at \,$x=\infty$\,. 

\begin{Theorem}
\label{T:inverseproblem}
For given spectral data \,$(X',\mathcal{S}',\infty,z)$\,, there exists an monomorphism of algebras 
\begin{gather*} 
	\Phi: \{f\in H^0(X',\mathcal{M})\,\mid\,\forall\, p\in X'\setminus\{\infty\}:\,f_p\in\mathcal{O}_{X',p}\}
	\;\longrightarrow\; \mathcal{A}(\C \setminus T)   \\
	f \;\longmapsto\; \Phi(f) \qquad \text{so that} \quad \Phi(f)\psi = f \cdot \psi \;, 
\end{gather*}
where \,$\psi$\, is the Baker-Akhiezer function corresponding to \,$\mathcal{S}'$\,. The two highest coefficients of \,$\Phi(f)$\, are constant. The image \,$A$\, of \,$\Phi$\, in \,$\mathcal{A}(\C \setminus T)$\, belongs to \,$\mathcal{R}$\,. 
\end{Theorem}

\begin{proof}
We conclude that for all $t_0\in{\mathbb C}\setminus T$ the sheaf $\mathcal{S}'\otimes \mathcal{L}_{1/z}(t_0)$ has a one-dimensional space of global sections on $X'$, and all non-trivial sections do not vanish at $\infty$. Therefore, this sheaf is isomorphic to generalised divisor \,$\mathcal{S}$\, which contains the sheaf of holomorphic functions $\Sh{O}_{X'}$. The support of the sheaf $\mathcal{S}/\Sh{O}_{X'}$ is contained in $X\setminus\{\infty\}$. Due to \cite[Theorem~8.8]{KLSS} there exists a unique {\em Baker-Akhiezer function} on $X\times\{t\in{\mathbb C}|\;t+t_0\not\in T\}$ corresponding to the one-dimensional family of sheaves
$$\mathcal{S}\otimes \mathcal{L}_{1/z}(t)\simeq
\mathcal{S}'\otimes \mathcal{L}_{1/z}(t+t_0)\text{ with }
t\in{\mathbb C}.$$
The differential operator $D^l$ acts on $\exp\left(2\pi\ci t/z\right)$ as the multiplication with $\left(2\pi\ci/z\right)^l$. Therefore the uniqueness of the Baker-Akhiezer function implies that for all meromorphic functions $f$ on $X'$ which are holomorphic on $X'\setminus\{\infty\}$, there exists a unique holomorphic differential operator $P=\Phi(f)$ on ${\mathbb C}\setminus T$, such that for all $y\in X'\setminus\{\infty\}$, the value $\psi(y,\cdot)$ of the  Baker-Akhiezer function solves the holomorphic differential equation $f(y)\psi(y,\cdot)=P\psi(y,\cdot).$ More precisely, the order of $P$ is equal to the degree of $f$. If $f=\sum_{i\geq -m}a_{i}z^{i}$ denotes the Laurent series of the function $f$ in some neighbourhood of $\infty$ with respect to the local parameter $z$, then the highest coefficient of $P$ is equal to $a_{-m}\,(2\pi\ci)^m$. Moreover, since the values of $\exp(-2\pi\ci t/z)\psi(x,t)$ at $\infty$ are equal to $1$, we have in a neighbourhood of $\infty$ the following equation of Laurent series with respect to $z$:
$$\exp(-2\pi\ci t/z)
\left(f(x)\psi(x,\cdot)-(2\pi\ci)^mD^m\psi(x,\cdot)\right)=
a_{m-1}z^{1-m}+O(z^{2-m}).$$
Therefore the coefficient of $D^{m-1}$ in $P$ is equal to $a_{1-m}(2\pi\ci)^{m-1}$.

If $\lambda$ and $\mu$ are two meromorphic functions on $X$, which are holomorphic on $X'\setminus\{\infty\}$, then the values of the  Baker-Akhiezer function at any element $x\in X'\setminus\{\infty\}$ yields a common solution of $(P-\lambda(x))\psi(x,\cdot)=0$ and $(Q-\mu(x))\psi(x,\cdot)=0$, where $P$ and $Q$ denotes the differential operators corresponding to $\lambda$ and $\mu$. Since the commutator of $P$ and $Q$ on $I$ is a differential operator of finite order, the same arguments as in the proof of Theorem~\ref{th:commuting operators} show that the commutator is equal to zero.

By construction of \,$\Phi$\,, the degree of \,$\Phi(f)$\, is equal to the degree of \,$f$\,, i.e.~the pole order of \,$f$\, at \,$\infty$\,. For \,$d>2g'-2$\, we have \,$H^1(X',\mathcal{O}_{d\cdot \infty})=0$\, by \cite[Corollary~6.6(c)]{KLSS}, and therefore by Riemann-Roch \,$H^0(X',\mathcal{O}_{d\cdot x1}) = d-g'+1$\,. It follows that for every \,$d > 2g'-1$\, there exists a meromorphic function \,$f$\, on \,$X'$\, with pole order \,$d$\, at \,$\infty$\,, and therefore \,$A$\, contains the element \,$\Phi(f)$\, of degree \,$d$\,. Lemma~\ref{L:d0} thus shows \,$A \in \mathcal{R}$\,. 
\end{proof}

\begin{Proposition}
Let spectral data \,$(X',\mathcal{S}',\infty,z)$\, be given, such that \,$X'$\, is endowed with an anti-holomorphic involution \,$\rho$\, so that \,$\infty$\, is a smooth point of the fixed point set of \,$\rho$\, and \,$\rho^* \overline{z}=-z$\, and \,$\rho_* \overline{\mathcal{S}'}=\mathcal{S}'$\,.

Then the restriction of the elements \,$\Phi(f)$\, with \,$\rho^* \overline{f}=f$\, to any connected component \,$I$\, of 
\,$i\R \setminus T$\, defines a subalgebra of \,$\mathcal{A}(I)$\, (case~1) which belongs to \,$\mathcal{R}$\,.    
\end{Proposition}

\begin{proof}
\,$\rho^*\bar{\psi}$\, is another function which satisfies the properties of the Baker-Akhiezer function \,$\psi$\, including the normalisation condition. Due to the uniqueness of the Baker-Akhiezer function we therefore have \,$\rho^*\bar{\psi}=\psi$\,. 
Therefore the differential operators \,$\Phi(f)$\, where \,$\rho^*\bar{f}=f$\, have real coefficients on any connected component \,$I$\, of  \,$i\R \setminus T$\,. For every \,$f\in H^0(X',\mathcal{M})$\, which is holomorphic on \,$X' \setminus \{\infty\}$\,, the degree of \,$\Phi(f+\rho^*\bar{f})$\, is the same as the degree of \,$\Phi(f)$\,. Therefore the subalgebra of \,$\mathcal{A}(I)$\, (case~1) belongs to \,$\mathcal{R}$\,. 
\end{proof}

The following theorem shows that the constructions of the direct problem in Section~\ref{Se:direct} and the inverse problem in Section~\ref{Se:inverse} are essentially inverse to each other. 

\begin{Theorem}
\label{T:inverseproblemsolution}
\begin{enumerate}
	\item Let \,$I$\, be a domain as in one of the three cases, \,$t_0\in I$\, and \,$A \in \mathcal{R}$\,, \,$(X'',\mathcal{S}'',\infty,z)$\, be the corresponding spectral data constructed in Theorem~\ref{th:direct general}, and \,$A(X'',\mathcal{S}'',\infty,z) \in \mathcal{R}$\, be the algebra corresponding to these spectral data by Theorem~\ref{T:inverseproblem}. Then \,$0\not\in T$\,,\,$I \subset \C \setminus (t_0+T)$, and the differential algebra \,$A(X'',\mathcal{S}'',\infty,z)$\, is isomorphic to \,$A$\, via the translation \,$t\mapsto t+t_0$\, and the restriction to \,$I$\,. 
	
	\item Let \,$(X',\mathcal{S}',\infty,z)$\, be spectral data as considered in Section~\ref{Se:inverse} and \,$t_0 \in I := \C\setminus T$\,. Let \,$A$\, be the algebra corresponding to \,$(X',\mathcal{S}',\infty,z)$\, as in Theorem~\ref{T:inverseproblem}. Then the spectral data corresponding to \,$(A,I,t_0)$\, by means of Theorem~\ref{th:direct general} are isomorphic to \,$(X',\mathcal{S}' \otimes \mathcal{L}_{1/z}(t_0),\infty,z)$\,.
\end{enumerate}
\end{Theorem}

\begin{proof}
We first consider the case where \,$A=\langle P,Q \rangle$\, is the algebra generated by the two commuting differential operators \,$P$\, and \,$Q$\, of co-prime order \,$m$ and \,$n$\,, respectively. We again consider the local parameter $z$ defined for large $|\lambda|$ such that $\lambda=(z/2\pi\ci)^{-m}$, $\mu=(z/2\pi\ci)^{-n}+O(z^{1-n})$ as $z\to 0$. Due to Equation~\eqref{eq:asymp-M}, the eigenfunction \,$\widehat{\psi}$\, of \,$M(\lambda)$\, has the asymptotic behaviour
$$ \widehat{\psi} = \Delta \cdot (1,\dotsc,1)^T \cdot (1+O(z)) $$
with
$$ \Delta := \mathrm{diag}\bigr( 1, 2\pi\ci/z, (2\pi\ci/z)^2, \dotsc, (2\pi\ci/z)^{m-1} \bigr)\; .  $$
Then by Equation~\eqref{eq:def-U} it follows
$$ \Delta \cdot U(\,\cdot\,,\lambda) \cdot \Delta^{-1} = 2\pi\ci/z \cdot \left( \begin{smallmatrix}
0 & 1 & 0 & \ldots & 0\\
0 & 0 & 1 & \ldots & 0\\
\vdots & \vdots & \vdots & \ddots & \vdots\\
1 & 0 & 0 & \ldots & 0
\end{smallmatrix} \right) - \left( \begin{smallmatrix} 0 & 0 & \ldots & 0 \\ 0 & 0 & \ldots & 0 \\ \vdots & \vdots & \ddots & \vdots \\ 0 & 0 & \ldots & \alpha_m \end{smallmatrix} \right) + O(z) \; . $$
Because \,$\widehat{\psi}$\, solves the differential equation \,$(D-U)\widehat{\psi}=0$\, and \,$\alpha_m$\, is constant, the asymptotic equation
$$ \Delta^{-1} \widehat{\psi} = (1,\dotsc,1,e^{-t\alpha_m})^T \cdot e^{2\pi\ci t/z} \cdot \bigr(1+O(z)\bigr) $$
follows. In particular \,$e^{-2\pi\ci t/z}\,\psi = e^{-2\pi\ci t/z}\,\widehat{\psi}_0$\, is holomorphic near \,$\infty$\, and equal to \,$1$\, there. By the definition of \,$\mathcal{S}'$\,, for all \,$t\in I$\,
$$ \psi(\,\cdot\,,t) = g(t,\lambda)\cdot \chi $$
is  a section of \,$\mathcal{S}'$\, on \,$X' \setminus \{\infty\}$\,. By uniqueness of the Baker-Akhiezer function it follows that \,$\psi$\, is equal to the Baker-Akhiezer function of \,$(X',\mathcal{S}',\infty,z)$\,. This proves (1) for \,$A = \langle P,Q \rangle$\,. 

For general \,$A \in \mathcal{R}$\,, we apply Theorem~\ref{th:direct general} and choose differential operators \,$P,Q \in A$\, of co-prime order. In this situation, for all \,$t\in \C$\, we have \,$H^0(X',\mathcal{S}'_{-\infty} \otimes \mathcal{L}_{1/z}(t)) \simeq H^0(X'',\mathcal{S}''_{-\infty} \otimes \mathcal{L}_{1/z}(t))$\,, because \,$\pi''_* (\mathcal{S}''_{-\infty}\otimes \mathcal{L}_{1/z}(t)) = \mathcal{S}'_{-\infty}\otimes \mathcal{L}_{1/z}(t)$\,. Moreover, that theorem shows that the Baker-Akhiezer function of \,$(X'',\mathcal{S}'',\infty,z)$\, is equal to the composition of \,$\pi'' \times \mathrm{id}_{\C\setminus T}$\, with the Baker-Akhiezer function of \,$(X',\mathcal{S}',\infty,z)$\,. This implies (1) for general \,$A$\,. 

In the situation of (2), choose two differential operators \,$P,Q \in A$\, of co-prime order. Let \,$(X'',\mathcal{S}'',\infty,z)$\, be the spectral data of \,$\langle P,Q \rangle$\, defined after Lemma~\ref{L:normalize-bounded}. Then there exists a one-sheeted covering \,$\pi': X' \to X''$\,. The arguments from the proof of (1) show that the Baker-Akhiezer function of \,$(X',\mathcal{S}',\infty,z)$\, is equal to the composition of \,$\pi' \times \mathrm{id}_{\C\setminus T}$\, with the Baker-Akhiezer function of \,$(X'',\mathcal{S}'',\infty,z)$\,. Because of Theorem~\ref{th:direct general}, this proves (2). 
\end{proof}

Theorem~\ref{th:direct general} shows that any \,$A\in \mathcal{R}$\, is isomorphic to the algebra \,$\mathcal{M}(X',\infty)$\, of meromorphic functions on the spectral curve \,$X'$\, of \,$A$\, with pole at most at \,$\infty$\,. In particular, if two  spectral curves \,$(X',\infty)$\, and \,$(X'',\infty)$\, with marked points are biholomorphic, then the corresponding algebras are also isomorphic. Let us now prove the converse.   

\begin{Proposition}
\label{P:isomorphic-algebras}
Let \,$(X',\infty)$\, and \,$(X'',\infty)$\, be two singular curves with a smooth point and isomorphic algebras \,$\mathcal{M}(X',\infty)$\, and \,$\mathcal{M}(X'',\infty)$\,. Then \,$(X',\infty)$\, is biholomorphic to \,$(X'',\infty)$\,. 
\end{Proposition}	

\begin{proof}
Due to Lemma~\ref{L:deg} the degrees of two elements of two algebras $\mathcal{M}(X',\infty)$ and $\mathcal{M}(X'',\infty)$, respectively, coincide if they are mapped onto each other by an isomorphism $\mathcal{M}(X',\infty)\simeq\mathcal{M}(X'',\infty)$. First we choose two elements $\lambda$ and $\mu$ of $\mathcal{M}(X',\infty)\simeq\mathcal{M}(X'',\infty)$ of co-prime order. Due to Theorem~\ref{th:commuting operators} there exists a polynomial $f$ with $f(\lambda,\mu)=0$. This equation defines a singular curve $X$ with smooth marked point $\infty$. Due to Theorem~\ref{th:direct general} both singular curves $X'$ and $X''$ are one-sheeted coverings $\pi':X'\to X$ and $\pi'':X''\to X$ of this curve.

Let us now show that for all $x\in X\setminus\{\infty\}$ the algebras $\mathcal{M}(X',\infty)\simeq\mathcal{M}(X'',\infty)$ generate the subrings $(\pi'_\ast(\mathcal{O}_{X'})_x$ and $(\pi''_\ast(\mathcal{O}_{X''})_x$, respectively. By symmetry it suffices to give the argument for $\mathcal{M}(X',\infty)$. As in~\cite{KLSS} we denote the direct image of the sheaf of holomorphic functions on the normalisation of $X$ by $\bar{\mathcal{O}}_X$. For \,$x\in X\setminus \{\infty\}$\, let $r_x$ be the radical $r_x=\{f\in\bar{\mathcal{O}}_{X,x}\mid f(x)=0\}$. Due to~\cite[Proposition~2.1]{KLSS} $(\pi'_\ast(\mathcal{O}_{X'}))_x\supset\mathcal{O}_{X,x}$ contains $r_x^n$ for some $n\in\mathbb{N}$. Since the multiplication is surjective from $r_x^n\times r_x^n$ onto $r_x^{2n}$, any choice of elements $f_1,\ldots,f_m$ of $(\pi'_\ast(\mathcal{O}_{X'}))_x\cap r_x$ which span $(\pi'_\ast(\mathcal{O}_{X'}))_x\cap r_x/r_x^{2n}$ define a surjective homomorphism $\C\{f_1,\ldots,f_m\}\to(\pi'_\ast(\mathcal{O}_{X'}))_x$. Let $\mathcal{S}'$ be the unique generalized divisor with support $\{\infty\}$ and degree $\deg(\mathcal{S}')=2g'-1+\dim((\pi'_\ast(\mathcal{O}_{X'}))_x/r_x^{2n})$. Let $\mathcal{S}$ be the unique subsheaf of $\mathcal{S}'$ which coincides on $X\setminus\{x\}$ with $\mathcal{S}'$ and with stalk $\mathcal{S}_x=r_x^{2n}$. It has the degree $\deg(\mathcal{S})=2g'-1$. By Serre duality, we have \,$H^1(X,\mathcal{S}')=H^1(X,\mathcal{S})=0$\,, and therefore the Riemann-Roch Theorem implies
\begin{equation*}
\dim(H^0(X',\mathcal{S}'))-\dim(H^0(X',\mathcal{S})) = \deg(\mathcal{S}')-\deg(\mathcal{S}) \; . 
\end{equation*}
Hence the natural projection of the subspace $H^0(X',\mathcal{S}')\subset\mathcal{M}(X',\infty)$ onto $(\pi'_\ast(\mathcal{O}_{X'}))_x/r_x^{2n}$ is surjective. Moreover, there exists $f_1,\ldots,f_m\in\mathcal{M}(X',\infty)$, which vanish at $x$ and induce a surjective homomorphism $\C\{f_1,\ldots,f_m\}\to(\pi'_\ast(\mathcal{O}_{X'}))_x$. this proves the claim. In particular, two points of the normalisation of $X$ belong to the same point of $X'$ if and only if all functions of $\mathcal{M}(X',\infty)$ take at both points the same values.

Consequently the sheaves $\pi'_\ast(\mathcal{O}_{X'})$ and $\pi''_\ast(\mathcal{O}_{X''})$ are isomorphic. This implies first that $X'$ and $X''$ are homeomorphic and second the sheaves $\mathcal{O}_{X'}$ and $\mathcal{O}_{X''}$ are isomorphic. Now \cite[Proposition~2.3]{KLSS} proves that $(X',\infty)$ and $(X'',\infty)$ are biholomorphic.
\end{proof}

In the case where the spectral curve has geometric genus zero, the commuting differential operators \,$P$\, and \,$Q$\, which generate the corresponding rank 1 algebra can be computed explicitly. We conclude this paper with an example of this computation. 

Let us consider \,$A \in \mathcal{R}$\,, generated by two commuting differential operators \,$P$\, and \,$Q$\, of co-prime degree \,$m$\, and \,$n$\,, respectively. We let \,$(X',\mathcal{S}',\infty,z)$\, be the spectral data corresponding to \,$A$\, and suppose that the spectral curve \,$X'$\, has geometric genus zero. The simplest possible case occurs when \,$m=2$\,, \,$n=3$\,, which we will investigate in the sequel.

Because \,$X'$\, has geometric genus zero, there exists a global coordinate of the normalisation \,$X$\, of \,$X'$\,, i.e.~a global meromorphic function \,$z$\, on \,$X$\, which is zero at \,$\infty$\, and nowhere else, so that the functions \,$\lambda$\, and \,$\mu$\, are given as \,$\lambda=p(z^{-1})$\, and \,$\mu=q(z^{-1})$\, in terms of polynomials \,$p$\, of degree \,$m=2$\, and \,$q$\, of degree \,$n=3$\,. By choosing the generating operators \,$P$\, and \,$Q$\, and the coordinate \,$z$\, suitably, one can achieve
\begin{equation}
\label{eq:example-pg}
\lambda=z^{-2} \quad\text{and}\quad \mu= z^{-3}+b_1\,z^{-1}
\end{equation}
with some \,$b_1\in \C$\,. Indeed, \,$z$\, can be chosen such that \,$p(z^{-1})=z^{-2}+a_0$\, with some \,$a_0\in \C$\,. After subtracting the constant \,$a_0$\, from \,$P$\, and normalising \,$Q$\,, we obtain \,$p(z^{-1})=z^{-2}$\, and $q(z^{-1})=z^{-3}+b_2z^{-2}+b_1z^{-1}+b_0$. By now subtracting \,$b_2 P + b_0$\, from \,$Q$\,, we obtain \eqref{eq:example-pg}. These \,$\lambda$\,, \,$\mu$\, satisfy the relation \,$f(\lambda,\mu)=0$\, with the polynomial \,$f(\lambda,\mu)$\, given by
$$ f(\lambda,\mu)=\mu^2-\lambda(\lambda+b_1)^2 = \mu^2-\lambda^3 - 2b_1\lambda^2-b_1^2\lambda \; . $$
The complex curve defined by the equation \,$f(\lambda,\mu)=0$\,, compactified by adding a smooth point at \,$\infty$\,, is hyperelliptic, and has exactly one singularity. This is a double point at \,$\lambda=-b_1$\, if \,$b_1\neq 0$\,, and a cusp at \,$\lambda=0$\, if \,$b_1=0$\,. Thus \,$X'$\, is either equal to this curve (and then has arithmetic genus \,$1$\,), or to its normalisation (and then has arithmetic genus \,$0$\,).

If \,$X'$\, has arithmetic genus zero, then \,$X'$\, is biholomorphic to the Riemann sphere and  the corresponding Baker-Akhiezer function is holomorphic outside \,$\infty$\,. The corresponding ordinary differential operators have constant coefficients and are therefore equal to 
\begin{align*}
P&=D^2 & Q&=D^3+b_1D\;. 
\end{align*}

If \,$X'$\, has arithmetic genus \,$1$\,, we at first consider the case \,$b_1\neq 0$\,. We will see that the case \,$b_1=0$\, can be treated via the same calculations by taking the limit. Here \,$X'$\, is a one-fold cover below the Riemann sphere, obtained by identifying the two points \,$z^{-1}=\pm c$\, with \,$c:=\sqrt{-b_1}$\, as a double point \,$z_0$\,. In this case the corresponding Baker-Akhiezer function is holomorphic except for a single order pole at \,$z=z_0$\, and an essential singularity at \,$\infty$\,, hence it is of the form 
$$\psi(z,t)=\exp(2\pi\ci tz^{-1})\frac{z^{-1}+d(t)}{z^{-1}-z_0^{-1}}
=\exp(2\pi\ci tz^{-1})\frac{z_0+d(t)z_0z}{z_0-z},$$
with a suitable function $d$ depending on $t$. Because \,$\psi(z,t)$\, has to take the same values at $z^{-1}=\pm c$, we have
$$\exp(2\pi\ci ct)\frac{c+d(t)}{c-z_0^{-1}}=
\exp(-2\pi\ci ct)\frac{-c+d(t)}{-c-z_0^{-1}},$$
whence it follows that $d(t)$ is given by
\begin{equation*}\begin{split}
d(t)
&=-c\frac{z_0^{-1}\cos(2\pi ct)+c\ci\sin(2\pi ct)}
         {c\cos(2\pi ct)+z_0^{-1}\ci\sin(2\pi ct)} \; . 
\end{split}\end{equation*}
Therefore the {\em Baker-Akhiezer function} is equal to
$$\psi(z,t)=\frac{\exp(2\pi\ci tz^{-1})}{z^{-1}-z_0^{-1}}
\left(z^{-1}-c\frac{z_0^{-1}\cos(2\pi ct)+c\ci\sin(2\pi ct)}
         {c\cos(2\pi ct)+z_0^{-1}\ci\sin(2\pi ct)}\right).$$
An explicit calculation shows that the Baker-Akhiezer function solves the differential equation
$$\frac{\partial^2}{\partial t^2} \psi(z,t)=
-4\pi^2\left(z^{-2}+\frac{2c^2(z_0^{-2}-c^2)}
                     {(c\cos(2\pi ct)+z_0^{-1}\ci\sin(2\pi ct))^2}
\right)\psi(z,t).$$
This shows that the operator $P$ corresponding to the function $\lambda=z^{-2}$ is given by
$$P=-1/(4\pi^2)D^2-\frac{2c^2(z_0^{-2}-c^2)}
                     {(c\cos(2\pi ct)+z_0^{-1}\ci\sin(2\pi ct))^2}.$$
We leave it to the reader to calculate the corresponding operator $Q$.

Finally we consider the limit $c\rightarrow 0$. In this case the Baker-Akhiezer function is equal to
$$\psi(z,t)=\frac{\exp(2\pi\ci tz^{-1})}{z^{-1}-z_0^{-1}}
\left(z^{-1}-\frac{z_0^{-1}}
         {1+2\pi\ci z_0^{-1}t}\right),$$
and $P$ is equal to
$$P=-1/(4\pi^2)D^2-\frac{2z_0^{-2}}
                     {(1+2\pi\ci z_0^{-1}t)^2}.$$
 
\end{document}